\documentclass{amsart}

\usepackage[T1]{fontenc}
\usepackage{lmodern}
\usepackage{geometry}
\usepackage{listings}
\usepackage{amsmath,amssymb,amsthm,amsfonts,mathtools}
\usepackage{hyperref}
\newcommand{\R}{\mathbb{R}}
\newcommand{\C}{\mathbb{C}}
\newcommand{\D}{\mathbb{D}}

\theoremstyle{plain}
\newtheorem{theorem}{Theorem}
\newtheorem{lemma}[theorem]{Lemma}
\newtheorem{proposition}[theorem]{Proposition}
\theoremstyle{remark}
\newtheorem{remark}[theorem]{Remark}

\title{The $L^2$-Norm of the Cauchy transform on circular annuli}
\author{David Kalaj}
\address{University of Montenegro, Faculty of Natural Sciences and
Mathematics, Cetinjski put b.b. 81000 Podgorica, Montenegro}
\email{davidk@ucg.ac.me}

\date{}
\begin{document}
\maketitle

\begin{abstract}
We compute the exact $L^2$ operator norm of the Cauchy transform
\[
(C_\Omega f)(z)=\frac1\pi\int_\Omega \frac{f(w)}{z-w}\,dA(w)
\]
on a circular annulus $\Omega=A(r,R)=\{r<|z|<R\}$.
Exploiting rotational symmetry and a Fourier mode decomposition,
we reduce the problem to a one--dimensional weighted Hardy operator and obtain
\[
\|C_{A(r,R)}\|_{L^2\to L^2}
=
\frac{2}{\sqrt{\mu_1^{ND}(r,R)}},
\]
where $\mu_1^{ND}(r,R)$ is the first eigenvalue of the Laplacian on $A(r,R)$
with Neumann condition on the inner boundary and Dirichlet condition on the outer boundary.
The extremizers are explicitly described in terms of Bessel functions.
\end{abstract}

\section{Introduction}

The Cauchy transform on a planar domain $\Omega\subset\C$,
\[
(C_\Omega f)(z)=\frac1\pi\int_\Omega \frac{f(w)}{z-w}\,dA(w),
\]
is a fundamental operator in complex analysis and elliptic PDE.  It gives a
canonical solution operator for the $\bar\partial$-equation in the sense that
\[
\bar\partial\,(C_\Omega f)=f\qquad\text{in }\Omega
\]
(in the distributional sense, and pointwise whenever $f$ is sufficiently
regular).  Consequently, sharp bounds for the operator norm
$\|C_\Omega\|_{L^2(\Omega)\to L^2(\Omega)}$ are closely connected to
quantitative $\bar\partial$ estimates and to spectral data of associated
elliptic boundary value problems.

On smoothly bounded connected domains one has the spectral lower bound
\[
\|C_\Omega\|_{L^2(\Omega)\to L^2(\Omega)}\ge \frac{2}{\sqrt{\lambda_1(\Omega)}},
\]
where $\lambda_1(\Omega)$ denotes the first Dirichlet eigenvalue of $-\Delta$
on $\Omega$; see \cite{AndersonKhavinsonLomonosov1992}.  More precisely, the
estimate can be expressed in the equivalent form
\begin{proposition}[Norm estimate for the Cauchy transform]
Let $\Omega\subset\R^2$ be a smoothly bounded domain. Then
\[
\left\|\frac{C_\Omega^{*}C_\Omega}{4}\right\|
=\frac14\,\|C_\Omega\|^2
\ge \frac{1}{\lambda_1(\Omega)},
\]
where $\lambda_1(\Omega)$ is the smallest positive eigenvalue for the boundary
value problem \emph{(1.2)} in $\Omega$.
\end{proposition}

For the unit disk $D=\{z\in\C:\ |z|<1\}$, Anderson and Hinkkanen computed the
exact $L^2$-operator norm.  If $j_{0,1}$ denotes the smallest positive zero of
the Bessel function $J_0$, then
\[
\|C_D\|_{L^2(D)\to L^2(D)}=\frac{2}{j_{0,1}},
\qquad\text{equivalently}\qquad
\frac14\,\|C_D\|^2=\frac{1}{j_{0,1}^2}.
\]
Since $\lambda_1(D)=j_{0,1}^2$, this shows that equality holds in the disk case:
\[
\frac14\,\|C_D\|^2=\frac{1}{\lambda_1(D)}.
\]
We refer to \cite{AndersonHinkkanen1989} for the sharp constant and explicit
extremizers.  In contrast, for multiply connected domains additional holomorphic
components may influence sharpness.

In this paper we compute the exact $L^2$-operator norm of $C_\Omega$ when
$\Omega$ is the circular annulus
\[
A(r,R)=\{z\in\C:\ r<|z|<R\}.
\]

Our main result identifies the norm with the first mixed eigenvalue of the
Laplacian (Neumann on \(|z|=r\), Dirichlet on \(|z|=R\)).

\begin{theorem}[Main result]\label{thm:intro-main}
Let \(0<r<R\). Then
\[
\|C_{A(r,R)}\|_{L^2\to L^2}=\frac{2}{\sqrt{\mu_1^{ND}(r,R)}},
\]
where \(\mu_1^{ND}(r,R)\) is the first eigenvalue of
\[
-\Delta U=\mu U \ \text{in }A(r,R),\qquad
\partial_\nu U=0 \ \text{on }|z|=r,\qquad
U=0 \ \text{on }|z|=R.
\]
Equivalently, \(\mu_1^{ND}(r,R)=\kappa_1^2\), where \(\kappa_1\) is the smallest
positive root of a Bessel determinant (see Theorem~\ref{thm:norm-annulus}).
\end{theorem}

The proof exploits rotational symmetry: a Fourier decomposition reduces
\(C_{A(r,R)}\) to a family of one-dimensional weighted Hardy-type operators on
radial profiles. A variational argument then yields a Sturm--Liouville problem
whose first eigenvalue controls the sharp operator norm; the Bessel determinant
arises from solving the radial eigenvalue equation explicitly.
The extremizers are given in terms of the corresponding mixed eigenfunction,
revealing a structural distinction between simply and doubly connected domains.
\subsection{Background and related work}
Norm bounds and spectral properties of the Cauchy transform and related
operators have been studied from several perspectives.
Mapping properties on bounded domains were initiated in
\cite{AndersonHinkkanen1989}, while a spectral approach to potential-theoretic
integral operators was developed in \cite{AndersonKhavinsonLomonosov1992}.
Sharp \(L^2\) estimates and their relation to spectral data appear in
\cite{ArazyKhavinson1992} and subsequent work of Dostani\'c
\cite{Dostanic1996,Dostanic1998PLMS,Dostanic2005IEOT}.
Connections with Poisson-type problems and quantitative PDE estimates were
studied in \cite{Kalaj2012AIM,Kalaj2012IEOT}, with further norm estimates for
Cauchy-type operators in \cite{ZhuKalaj2020JFA} and sharp \(L^p\)-theory on the
disk in \cite{KalajMelentijevicZhu2022JFA}.
Related spectral questions on Bergman-type spaces and in doubly connected
geometry can be found in \cite{Vujadinovic2016JMAA,Vujadinovic2025PA}.

\section{The proof of main results}
\begin{lemma}[Exact action on a single Fourier mode]\label{lem:mode-action}
Let $A=A(r,R)=\{z:\ r<|z|<R\}$ and
\[
(C_A f)(z)=\frac1\pi\int_A \frac{f(w)}{z-w}\,dA(w).
\]
Fix $m\in\mathbb Z$ and write $w=\rho e^{i\phi}$, $z=\rho_0 e^{i\theta}$.
If
\[
f(w)=g(\rho)e^{im\phi},
\]
then $C_A f$ lies in the $(m-1)$-st angular mode and is given by
\[
(C_A f)(\rho_0 e^{i\theta})
=
\begin{cases}
-2\,\rho_0^{m-1}e^{i(m-1)\theta}\displaystyle\int_{\rho_0}^{R} g(\rho)\,\rho^{1-m}\,d\rho,
& m\ge 1,\\[2.0ex]
\ \ 2\,\rho_0^{m-1}e^{i(m-1)\theta}\displaystyle\int_{r}^{\rho_0} g(\rho)\,\rho^{1-m}\,d\rho,
& m\le 0.
\end{cases}
\]
\end{lemma}

\begin{proof}
Use the geometric-series expansions of $(z-w)^{-1}$:
\[
\frac1{z-w}=
\frac1z\sum_{n\ge 0}\Big(\frac{w}{z}\Big)^n,\quad |w|<|z|,
\qquad
\frac1{z-w}=
-\frac1w\sum_{n\ge 0}\Big(\frac{z}{w}\Big)^n,\quad |w|>|z|.
\]
Insert $w=\rho e^{i\phi}$, $z=\rho_0 e^{i\theta}$ and multiply by $e^{im\phi}$.
Integrating in $\phi$ over $[0,2\pi]$ kills all terms except the unique one
whose angular frequency is $0$; this occurs only when $n=-m$ in the case
$|w|<|z|$ (hence $m\le 0$), and only when $n=m-1$ in the case $|w|>|z|$
(hence $m\ge 1$). The stated formulas follow after the remaining radial
integration.
\end{proof}

\begin{lemma}[Weighted Hardy reduction for $m\ge 1$]\label{lem:hardy-m}
Fix $m\ge 1$ and set
\[
(T_m g)(\rho):=-2\,\rho^{m-1}\int_{\rho}^{R} g(t)\,t^{1-m}\,dt,
\qquad r<\rho<R.
\]
Then $\|C_A\|_{H_m\to H_{m-1}}=\|T_m\|$ when both $H_m,H_{m-1}$ are equipped
with the $L^2(A)$ norm.
Moreover, if $F(\rho):=\int_{\rho}^{R} g(t)\,t^{1-m}\,dt$ (so $F(R)=0$ and
$F'(\rho)=-g(\rho)\rho^{1-m}$), then
\[
\frac{\|T_m g\|_{L^2((r,R),\rho\,d\rho)}^2}{\|g\|_{L^2((r,R),\rho\,d\rho)}^2}
=
\frac{4\int_r^R |F(\rho)|^2\,\rho^{2m-1}\,d\rho}{\int_r^R |F'(\rho)|^2\,\rho^{2m-1}\,d\rho}.
\]
Hence
\[
\|T_m\|=\frac{2}{\sqrt{\mu_m}},
\]
where $\mu_m$ is the first eigenvalue of
\begin{equation}\label{eq:SL-m}
-(\rho^{2m-1}F'(\rho))'=\mu\,\rho^{2m-1}F(\rho),\quad r<\rho<R,
\qquad F'(r)=0,\ \ F(R)=0.
\end{equation} 
\end{lemma}

\begin{proof}
The mode formula in Lemma~\ref{lem:mode-action} shows that the restriction
of $C_A$ to mode $m$ is exactly $T_m$ on the radial profile.
The $L^2(A)$ norm of $g(\rho)e^{im\theta}$ is $2\pi\int_r^R |g(\rho)|^2\rho\,d\rho$,
and similarly for the output mode, so the operator norm reduces to $\|T_m\|$.
The Rayleigh quotient in $F$ is immediate from $T_m g = -2\rho^{m-1}F$ and
$g=-\rho^{m-1}F'$, and the Euler--Lagrange equation yields \eqref{eq:SL-m}
with the natural boundary condition $F'(r)=0$. Here are the details of the last step: % --- Step (3): deriving the Sturm--Liouville problem and boundary conditions ---

Let \(w(\rho):=\rho^{2m-1}\).  From the previous step we arrive at the Rayleigh
quotient
\[
\frac{\|T_m g\|_{L^2((r,R),\rho\,d\rho)}^2}{\|g\|_{L^2((r,R),\rho\,d\rho)}^2}
=
\frac{4\int_r^R |F(\rho)|^2\,w(\rho)\,d\rho}{\int_r^R |F'(\rho)|^2\,w(\rho)\,d\rho}.
\]
Thus maximizing the left-hand side is equivalent to minimizing
\[
\mathcal{R}[F]:=\frac{\int_r^R |F'(\rho)|^2\,w(\rho)\,d\rho}
{\int_r^R |F(\rho)|^2\,w(\rho)\,d\rho}
\quad\text{over }F\not\equiv 0\text{ with }F(R)=0.
\]
Introduce a Lagrange multiplier \(\mu\) and consider
\[
\mathcal{J}[F]:=\int_r^R\Big(|F'(\rho)|^2-\mu\,|F(\rho)|^2\Big)\,w(\rho)\,d\rho,
\qquad F(R)=0.
\]
Take a variation \(F_\varepsilon=F+\varepsilon \varphi\) with \(\varphi(R)=0\).
Then
\[
\left.\frac{d}{d\varepsilon}\right|_{\varepsilon=0}\mathcal{J}[F_\varepsilon]
=
2\int_r^R \big(F'(\rho)\varphi'(\rho)-\mu F(\rho)\varphi(\rho)\big)\,w(\rho)\,d\rho.
\]
Integrating by parts gives
\[
2\int_r^R F' \varphi'\,w\,d\rho
=
2\Big[w(\rho)F'(\rho)\varphi(\rho)\Big]_{r}^{R}
-2\int_r^R (w(\rho)F'(\rho))'\,\varphi(\rho)\,d\rho,
\]
hence
\[
\left.\frac{d}{d\varepsilon}\right|_{\varepsilon=0}\mathcal{J}[F_\varepsilon]
=
2\Big[w(\rho)F'(\rho)\varphi(\rho)\Big]_{r}^{R}
-2\int_r^R\Big((wF')' + \mu w F\Big)\varphi(\rho)\,d\rho.
\]
Since \(\varphi(R)=0\), the boundary term at \(R\) vanishes.  At \(\rho=r\) the
value \(F(r)\) is free, so \(\varphi(r)\) is arbitrary; therefore the remaining
boundary term forces
\[
w(r)F'(r)=0 \quad\Longrightarrow\quad F'(r)=0
\qquad(\text{since }w(r)=r^{2m-1}>0).
\]
Because \(\varphi\) is arbitrary in \((r,R)\), the integral term yields the
Euler--Lagrange equation
\[
(w(\rho)F'(\rho))' + \mu\,w(\rho)F(\rho)=0
\quad\Longleftrightarrow\quad
-(\rho^{2m-1}F'(\rho))'=\mu\,\rho^{2m-1}F(\rho),
\]
together with the boundary conditions \(F'(r)=0\) and \(F(R)=0\).  This is
exactly the Sturm--Liouville problem \eqref{eq:SL-m}.

\end{proof}
\paragraph{Bessel functions and the characteristic equation.}
The symbols \(J_k\) and \(Y_k\) denote the Bessel functions of the first and
second kind (Neumann functions) of order \(k\). They form a fundamental pair of
linearly independent solutions of Bessel's differential equation
\[
x^2 y''(x)+x y'(x)+\bigl(x^2-k^2\bigr)y(x)=0.
\]
The function \(J_k\) is the solution that is finite at \(x=0\), while \(Y_k\) is
the second independent solution, which is singular at \(x=0\) for \(k\ge 0\).
In the present annulus problem, separation of variables reduces the mixed
boundary-value eigenproblem for the Laplacian to a radial Bessel ODE. Imposing
the Neumann condition at \(|z|=r\) and the Dirichlet condition at \(|z|=R\)
produces a nontrivial solution precisely when a certain Bessel determinant
vanishes; for the principal mode this takes the form
\[
J_1(\kappa r)Y_0(\kappa R)-Y_1(\kappa r)J_0(\kappa R)=0,
\]
and the smallest positive root \(\kappa\) determines the sharp norm constant via
\(\|C_{A(r,R)}\|_{2\to2}=2/\kappa\).

\begin{lemma}[Bessel form for $m\ge 1$]\label{lem:bessel-m}
Let $m\ge 1$ and write $\mu_m=\kappa_{m,1}^2$ with $\kappa_{m,1}>0$.
Then $\kappa_{m,1}$ is the smallest positive solution of
\begin{equation}\label{eq:kappa-m}
J_m(\kappa r)\,Y_{m-1}(\kappa R)-Y_m(\kappa r)\,J_{m-1}(\kappa R)=0,
\end{equation}
and
\[
\|C_A\|_{H_m\to H_{m-1}}=\|T_m\|=\frac{2}{\kappa_{m,1}}.
\]
In particular, for $m=1$ one gets \eqref{eq:kappa-m} as
\[
J_1(\kappa r)\,Y_0(\kappa R)-Y_1(\kappa r)\,J_0(\kappa R)=0.
\]
\end{lemma}

\begin{proof}
Set $U(\rho):=\rho^{m-1}F(\rho)$. Then \eqref{eq:SL-m} is equivalent to the
Bessel-type equation
\[
U''(\rho)+\frac1\rho U'(\rho)+\Big(\kappa^2-\frac{(m-1)^2}{\rho^2}\Big)U(\rho)=0,
\]
so $U(\rho)=aJ_{m-1}(\kappa\rho)+bY_{m-1}(\kappa\rho)$.
The boundary condition $F(R)=0$ becomes $U(R)=0$.
The condition $F'(r)=0$ becomes $rU'(r)=(m-1)U(r)$, which is equivalent to
$aJ_m(\kappa r)+bY_m(\kappa r)=0$ by the standard identity
$xJ'_{m-1}(x)-(m-1)J_{m-1}(x)=-xJ_m(x)$ (and the analogous one for $Y$).
The existence of a nontrivial $(a,b)$ gives \eqref{eq:kappa-m}.
Finally, Lemma~\ref{lem:hardy-m} gives $\|T_m\|=2/\sqrt{\mu_m}=2/\kappa_{m,1}$.
\end{proof}

\begin{lemma}\label{lem:monotone}
For $m\ge 1$, let $\mu_m$ be the first eigenvalue of \eqref{eq:SL-m}.
Then $\mu_1\le \mu_m$ for all $m\ge 1$, and hence
\[
\sup_{m\ge 1}\|C_A\|_{H_m\to H_{m-1}}=\|C_A\|_{H_1\to H_0}=\frac{2}{\sqrt{\mu_1}}
=\frac{2}{\kappa_{1,1}}.
\]
\end{lemma}

\begin{proof}
Rewrite the quotient in Lemma~\ref{lem:hardy-m} using $U(\rho)=\rho^{m-1}F(\rho)$:
a direct computation gives
\[
\int_r^R |F'(\rho)|^2\rho^{2m-1}\,d\rho
=
\int_r^R\Big(|U'(\rho)|^2\rho+\frac{(m-1)^2}{\rho}|U(\rho)|^2\Big)\,d\rho
+(m-1)|U(r)|^2,
\]
while
\[
\int_r^R |F(\rho)|^2\rho^{2m-1}\,d\rho=\int_r^R |U(\rho)|^2\rho\,d\rho.
\]
Therefore $\mu_m$ admits the variational characterization
\[
\mu_m
=\inf_{\substack{U\in H^1(r,R),\\ U(R)=0,\ U\neq 0}}
\frac{\displaystyle
\int_r^R\Big(|U'|^2\rho+\frac{(m-1)^2}{\rho}|U|^2\Big)\,d\rho
+(m-1)|U(r)|^2}
{\displaystyle\int_r^R |U|^2\rho\,d\rho}.
\]
For $m=1$ the numerator becomes $\int_r^R |U'|^2\rho\,d\rho$ (no potential term
and no boundary term). Hence for every admissible $U$,
\[
\frac{\int_r^R\Big(|U'|^2\rho+\frac{(m-1)^2}{\rho}|U|^2\Big)\,d\rho+(m-1)|U(r)|^2}
{\int_r^R |U|^2\rho\,d\rho}
\ \ge\
\frac{\int_r^R |U'|^2\rho\,d\rho}{\int_r^R |U|^2\rho\,d\rho}.
\]
Taking the infimum over the same class $\{U\in H^1:U(R)=0\}$ yields $\mu_m\ge \mu_1$.
\end{proof}

\begin{theorem}[Sharp \(L^2\)-norm on a circular annulus]\label{thm:norm-annulus}
Let \(0<r<R\), \(A=A(r,R)=\{z\in\C:\ r<|z|<R\}\), and
\[
(C_A f)(z)=\frac1\pi\int_A \frac{f(w)}{z-w}\,dA(w),\qquad f\in L^2(A).
\]
Then
\[
\|C_A\|_{L^2(A)\to L^2(A)}=\frac{2}{\kappa_1},
\]
where \(\kappa_1>0\) is the smallest positive solution of
\begin{equation}\label{eq:kappa1}
J_1(\kappa r)\,Y_0(\kappa R)-Y_1(\kappa r)\,J_0(\kappa R)=0 .
\end{equation}
\end{theorem}

\begin{proof}
By Lemmas~\ref{lem:mode-action} and \ref{lem:hardy-m},
\[
\|C_A\|=\sup_{m\in\mathbb Z}\|C_A\|_{H_m\to H_{m-1}},
\qquad
\|C_A\|_{H_m\to H_{m-1}}=
\begin{cases}
\|T_m\|,& m\ge 1,\\
\|T_{1-m}\|,& m\le 0,
\end{cases}
\]
and for \(m\ge 1\),
\[
\|T_m\|=\frac{2}{\sqrt{\mu_m}},
\]
where \(\mu_m\) is the first eigenvalue of \eqref{eq:SL-m}.
Lemma~\ref{lem:monotone} gives \(\mu_m\ge \mu_1\) for all \(m\ge 1\), hence the
supremum is attained at \(m=1\) and
\[
\|C_A\|=\|T_1\|=\frac{2}{\sqrt{\mu_1}}.
\]
Finally, Lemma~\ref{lem:bessel-m} (with \(m=1\)) yields \(\mu_1=\kappa_1^2\), where
\(\kappa_1\) is the smallest positive root of \eqref{eq:kappa1}. Therefore
\(\|C_A\|=2/\kappa_1\).

Moreover, if \(F>0\) is the first eigenfunction of \eqref{eq:SL-m} for \(m=1\) and
\(f(\rho e^{i\theta})=-F'(\rho)e^{i\theta}\), then Lemma~\ref{lem:mode-action} gives
\[
(C_A f)(\rho e^{i\theta})=2F(\rho),
\]
so \(f\) is an extremizer, see remark below for details.
\end{proof}
\begin{remark}[An explicit extremizer and the derivative computation]\label{rem:extremizer-annulus}
Let \(\kappa_1>0\) be the smallest positive solution of \eqref{eq:kappa1}. Define
\[
F(\rho):=c\Big( Y_0(\kappa_1 R)\,J_0(\kappa_1\rho)-J_0(\kappa_1 R)\,Y_0(\kappa_1\rho)\Big),
\qquad r<\rho<R,
\]
where \(c\neq 0\) is chosen so that \(F>0\) on \((r,R)\). Then \(F(R)=0\) holds by
construction. Moreover, using the identities \(J_0'(x)=-J_1(x)\) and
\(Y_0'(x)=-Y_1(x)\) together with the chain rule, we compute
\[
\begin{aligned}
F'(\rho)
&=c\Big( Y_0(\kappa_1 R)\,\frac{d}{d\rho}J_0(\kappa_1\rho)
        -J_0(\kappa_1 R)\,\frac{d}{d\rho}Y_0(\kappa_1\rho)\Big)\\
&=c\,\kappa_1\Big( Y_0(\kappa_1 R)\,J_0'(\kappa_1\rho)
        -J_0(\kappa_1 R)\,Y_0'(\kappa_1\rho)\Big)\\
&=c\,\kappa_1\Big(J_0(\kappa_1 R)\,Y_1(\kappa_1\rho)-Y_0(\kappa_1 R)\,J_1(\kappa_1\rho)\Big).
\end{aligned}
\]
In particular,
\[
F'(r)=0
\quad\Longleftrightarrow\quad
J_1(\kappa_1 r)\,Y_0(\kappa_1 R)-Y_1(\kappa_1 r)\,J_0(\kappa_1 R)=0,
\]
so \(\kappa_1\) is precisely the parameter for which the boundary conditions
\(F'(r)=0\) and \(F(R)=0\) are satisfied. Consequently,
\[
f(\rho e^{i\theta})=-F'(\rho)\,e^{i\theta}
\]
is an extremizer for \(\|C_A\|\) in Theorem~\ref{thm:norm-annulus}.
\end{remark}

% Remark: The case m<=0 is analogous (one gets a Hardy-type operator with the integral from r to rho),
% and the same conclusion holds: no negative mode yields a larger singular value than the m=1 block.
% Thus the full operator norm is 2/kappa_{1,1} with kappa_{1,1} determined by J_1(kr)Y_0(kR)-Y_1(kr)J_0(kR)=0.

\begin{remark}[On the graph of \(2/\kappa_{1,1}(r)\) for \(R=1\) and scaling in \(R\)]\label{rem:graph-scaling}
Theorem~\ref{thm:norm-annulus} gives
\[
\|C_{A(r,R)}\|_{2\to 2}=\frac{2}{\kappa_{1,1}(r,R)},
\]
where \(\kappa_{1,1}(r,R)\) is the smallest positive root of the transcendental
Bessel determinant \eqref{eq:kappa1}.  In general there is no closed-form
expression for \(\kappa_{1,1}(r,R)\) in terms of elementary functions, so an
\emph{explicit} formula for the function \(r\mapsto 2/\kappa_{1,1}(r,R)\) is not
available.  Nevertheless, for each fixed \(r\in(0,1)\) one can compute
\(\kappa_{1,1}(r,1)\) numerically by solving
\[
D_r(\kappa):=J_1(\kappa r)Y_0(\kappa)-Y_1(\kappa r)J_0(\kappa)=0
\]
for its smallest positive root (e.g.\ by bisection, Newton's method, or standard
special-function solvers).  This produces the graph of
\[
r\longmapsto \frac{2}{\kappa_{1,1}(r,1)},\qquad 0<r<1 (\text{see Figure~1}).
\]

It is enough to plot the case \(R=1\), since the problem has a simple scaling.
Indeed, under the dilation \(z=R\zeta\) the annulus \(A(r,R)\) is mapped onto
\(A(r/R,1)\), and a direct change of variables shows
\[
\|C_{A(r,R)}\|_{2\to 2}=R\,\|C_{A(r/R,1)}\|_{2\to 2}.
\]
Equivalently, writing \(a=r/R\in(0,1)\) and setting \(x=\kappa R\) in
\eqref{eq:kappa1}, one sees that the root depends only on the ratio \(a\):
\[
\kappa_{1,1}(r,R)=\frac{1}{R}\,\kappa_{1,1}(a,1),
\qquad a=\frac{r}{R}.
\]
Hence the graph for \(R=1\) is universal; all other radii \(R\) are obtained by
the above scaling. In the limiting case \(r=0\) and \(R=1\), the annulus \(A(r,R)\) becomes the unit disk.
Then \(\kappa_{1,1}(0,1)\) coincides with the smallest positive zero of \(J_0\).
Denoting this zero by \(\alpha=j_{0,1}\approx 2.4048255577\), we have
\[
\kappa_{1,1}(0,1)=\alpha,
\qquad\text{and hence}\qquad
\frac{2}{\kappa_{1,1}(0,1)}=\frac{2}{\alpha}\approx 0.8319.
\]
Consequently,
\[
\|C_{\D}\|_{L^2(\D)\to L^2(\D)}=\frac{2}{j_{0,1}},
\]
a value already obtained by Anderson and Hinkkanen \cite{AndersonHinkkanen1989}, (compare Figure 1 and  Figure~2).
\end{remark}

\begin{figure}[htbp]
  \centering
  \includegraphics[width=0.75\textwidth]{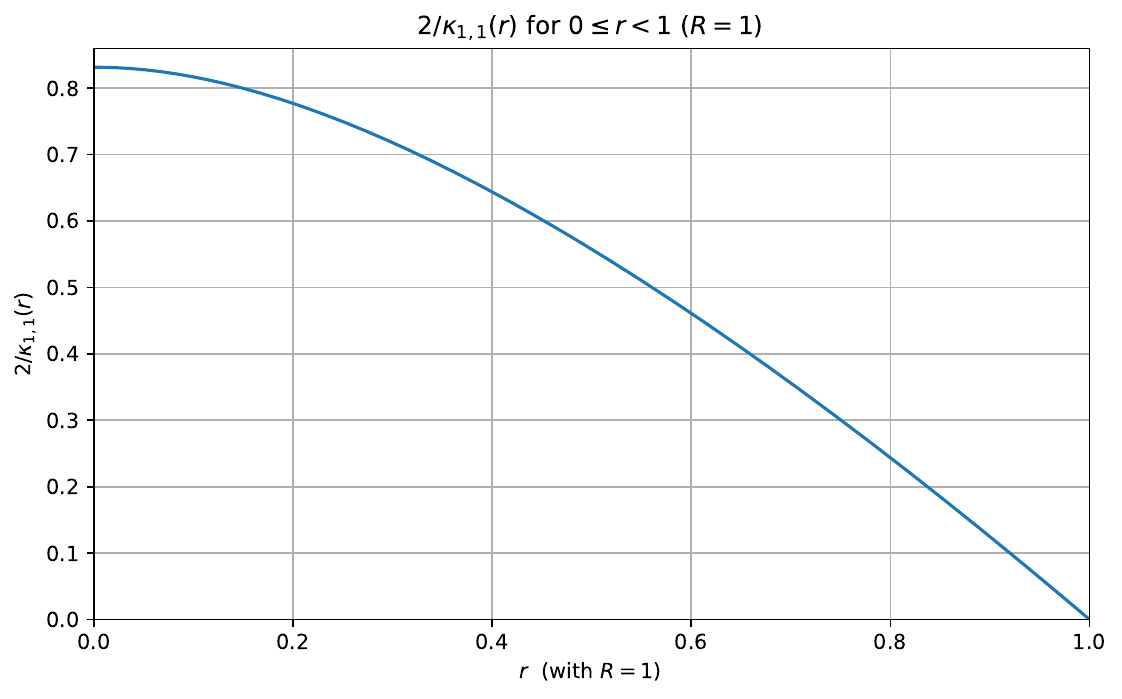}
  \caption{$2/\kappa_{1,1}(r)$ for $R=1$, $0\le r<1$.}
  \label{fig:two-over-kappa11}
\end{figure}

\section{Dirichlet--Cauchy operator and its sharp $L^2$ norm}

Fix an annulus with radii $0<r<R$.  For $g\in L^2(A)$, let $u$ be the (weak) solution of the zero--Dirichlet Poisson problem
\begin{equation}\label{eq:PoissonDirichlet}
\begin{cases}
-\Delta u=g & \text{in }A,\\
u=0 & \text{on }\partial A,
\end{cases}
\end{equation}
and define the \emph{Dirichlet--Cauchy operator} by
\begin{equation}\label{eq:DirichletCauchy}
\mathcal C_A[g](z):=\partial_z u(z),\qquad \partial_z=\tfrac12(\partial_x-i\partial_y).
\end{equation}
Equivalently, in terms of the Green function $G_A(\cdot,\tau)$ for \eqref{eq:PoissonDirichlet},
\begin{equation}\label{eq:GreenRep}
u(z)=\int_A G_A(z,\tau)\,g(\tau)\,dA(\tau),
\qquad
\mathcal C_A[g](z)=\int_A \partial_z G_A(z,\tau)\,g(\tau)\,dA(\tau).
\end{equation}

Let $\lambda_1(A)$ be the first Dirichlet eigenvalue of $-\Delta$ on $A$ and set
\begin{equation}\label{eq:k1def}
k_1:=\sqrt{\lambda_1(A)}.
\end{equation}
In the planar case, $k_1$ is characterized as the smallest positive root of the Bessel cross--product equation
\begin{equation}\label{eq:kEqGeneral}
J_0(kr)\,Y_0(kR)-J_0(kR)\,Y_0(kr)=0 .
\end{equation}

\begin{theorem}[\cite{Kalaj2012IEOT}]\label{thm:KalajNorm}
The Dirichlet--Cauchy operator \eqref{eq:DirichletCauchy} extends boundedly on $L^2(A)$ and satisfies the sharp estimate
\begin{equation}\label{eq:KalajNorm}
\|\mathcal C_A\|_{L^2(A)\to L^2(A)}=\frac{2}{k_1}.
\end{equation}
\end{theorem}

\begin{remark}[Normalization]\label{rem:Normalization}
The constant stated in \cite{Kalaj2012IEOT} appears as $\frac{1}{\sqrt{2}\,k_1}$.
This is a normalization effect: in complex notation one has
\[
\Delta=4\,\partial_z\partial_{\bar z},
\]
so writing the Poisson equation as $u_{z\bar z}=g$ corresponds to $\Delta u=4g$.
In addition, some authors define the transform as $\partial_z u$ while others use $2\,\partial_z u$
(or estimate $\nabla u$ instead of $\partial_z u$), which rescales the operator norm by fixed factors.
After translating between these conventions, the sharp constants are consistent.
\end{remark}

When $R=1$, the parameter $k_1=k_1(r)$ is the smallest positive solution of
\[
J_0(k)\,Y_0(rk)-Y_0(k)\,J_0(rk)=0,
\]
and hence the graph of the sharp norm \eqref{eq:KalajNorm} is obtained by plotting $2/k_1(r)$ for $r\in(0,1)$.

\begin{figure}[h]
\centering
\includegraphics[width=0.85\linewidth]{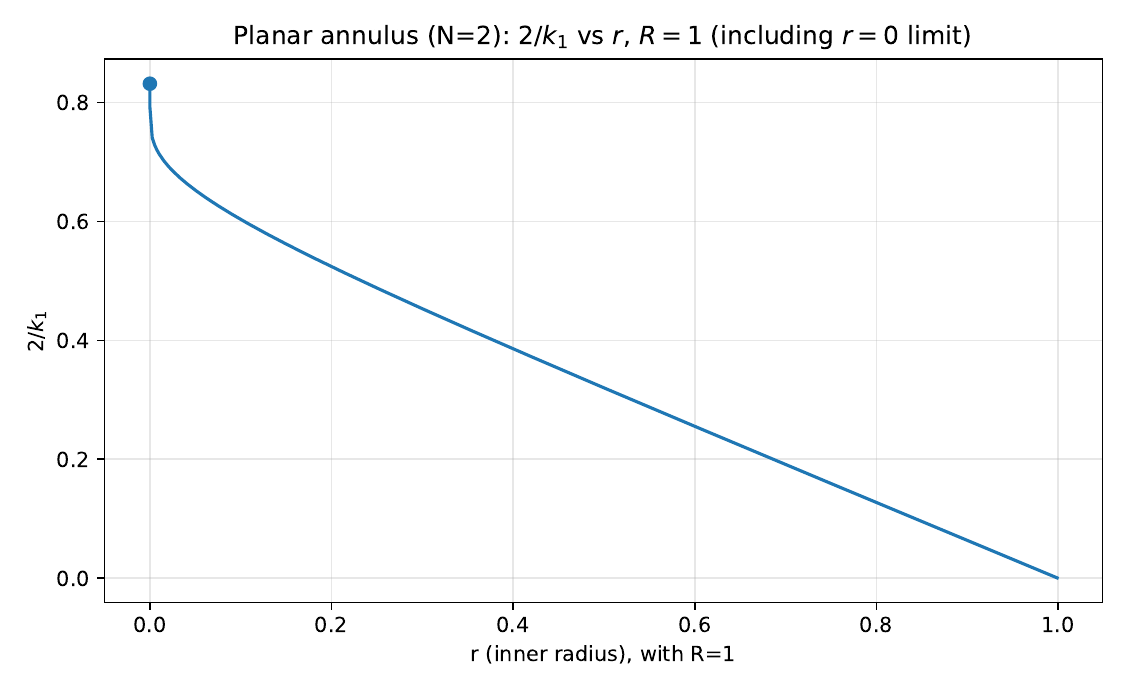}
\caption{Plot of the normalized quantity $2/k_1(r)$ for $R=1$ and $0<r<1$.}
\end{figure}
\section*{Acknowledgements}
The  author gratefully acknowledges financial support from the Ministry of Education, Science and Innovation of Montenegro through the grants \emph{``Mathematical Analysis, Optimisation and Machine Learning''} and \emph{``Complex-analytic and geometric techniques for non-Euclidean machine learning: theory and applications.''} He is also grateful to D.~Khavinson and P.~Melentijevi\'c for useful comments and encouragement.
 \paragraph{Competing interests.}
The author declares that there are no competing interests.

\paragraph{\bf Data availability.}
Data sharing is not applicable to this article as no datasets were generated or analysed during the current study.

\paragraph{\bf Ethical approval.}
Not applicable.

\paragraph{\bf Consent to participate.}
Not applicable.

\paragraph{\bf Consent for publication.}
Not applicable.

\end{document}